\documentclass[a4paper,12pt,oneside]{amsart}

\usepackage[applemac]{inputenc}
\usepackage[british]{babel}
\usepackage{fancyhdr}
\usepackage{geometry}
\usepackage{setspace}
\usepackage{amssymb}
\usepackage{bbm}
\usepackage[all]{xy}
\usepackage{color}\definecolor{citation}{rgb}{0.2,0.3,1}
\usepackage[final,colorlinks,citecolor=citation,pagebackref,linktocpage]{hyperref}

\newcommand{\leadingzero}[1]{\ifnum #1<10 0\the#1\else\the#1\fi}
\AtBeginDocument{\renewcommand{\today}{\the\year\leadingzero{\month}\leadingzero{\day}}}

\newlength{\aufzleft}
\newenvironment{mylist}[1]{\begin{list}{}{\setlength{\listparindent}{0pt}\setlength{\itemsep}{\topsep}\setlength{\labelwidth}{#1ex}\setlength{\aufzleft}{\labelsep}\addtolength{\aufzleft}{\labelwidth}\setlength{\leftmargin}{\aufzleft}}}{\end{list}}
\newenvironment{aufz}{\begin{mylist}{3.2}}{\end{mylist}}

\newtheoremstyle{par}{1ex}{2ex}{\rm}{}{\bfseries}{}{0.8em}{\thmnumber{(#2)}}
\newtheoremstyle{thm}{1ex}{2ex}{\itshape}{}{\bfseries}{}{0.9em}{\thmnumber{(#2)}\thmname{ #1}\thmnote{ (#3)}}
\newtheoremstyle{ex}{1ex}{2ex}{\rm}{}{\bfseries}{}{0.8em}{\thmnumber{(#2)}\thmname{ #1}}

\theoremstyle{par}
\newtheorem{no}{}[section]

\theoremstyle{thm}
\newtheorem{lemma}[no]{Lemma}
\newtheorem{prop}[no]{Proposition}
\newtheorem{cor}[no]{Corollary}
\newtheorem{thm}[no]{Theorem}

\theoremstyle{ex}
\newtheorem{exas}[no]{Examples}

\newcommand{\dfgl}{\mathrel{\mathop:}=}
\newcommand{\ph}{\varphi}
\newcommand{\N}{\mathbbm{N}}
\newcommand{\Z}{\mathbbm{Z}}
\newcommand{\Id}{\mathrm{Id}}
\newcommand{\hm}[3]{{\rm Hom}_{#1}(#2,#3)}
\newcommand{\ab}{{\sf Ab}}
\newcommand{\ann}{{\sf Ann}}
\newcommand{\md}{{\sf Mod}}
\newcommand{\C}{\sf{C}}
\newcommand{\D}{\sf{D}}
\newcommand{\ps}{_{[\psi]}}
\newcommand{\lph}{_{(\!(\ph)\!)}}
\newcommand{\ia}{\mathfrak{a}}
\renewcommand{\P}{\mathcal{P}}

\DeclareMathOperator{\ke}{Ker}
\DeclareMathOperator{\im}{Im}
\DeclareMathOperator{\degsupp}{degsupp}
\DeclareMathOperator{\card}{Card}
\DeclareMathOperator{\ob}{Ob}
\DeclareMathOperator{\dimp}{dp}
\DeclareMathOperator{\dimi}{di}
\DeclareMathOperator{\dimpl}{dpl}
\DeclareMathOperator{\rk}{rk}

\newcommand{\hme}{^{\rm hom}}
\DeclareMathOperator{\nzd}{Nzd}
\DeclareMathOperator{\zd}{Zd}
\DeclareMathOperator{\nil}{Nil}
\newcommand{\ip}{\mathfrak{p}}
\DeclareMathOperator{\spec}{Spec}
\DeclareMathOperator{\var}{Var}
\newcommand{\R}{\mathbbm{R}}
\DeclareMathOperator{\diff}{Diff}
\newcommand{\loccit}{\textit{loc.\,cit.\ }}

\begin{document}

\title{On certain properties and invariants\\of graded rings and modules}
\author{Fred Rohrer}

\address{Grosse Grof 9, 9470 Buchs, Switzerland}
\email{fredrohrer@math.ch}

\begin{abstract}
The behaviour under coarsening functors of simple, entire, or reduced graded rings, of free graded modules over principal graded rings, of superfluous monomorphisms and of homological dimensions of graded modules, as well as adjoints of degree restriction functors, are investigated.
\end{abstract}

\maketitle\thispagestyle{fancy}


\section*{Introduction}

Rings and modules graded by a group are ubiquitous in algebra. It is thus useful to compare them with their ungraded companions. N\u{a}st\u{a}sescu and Van Oystaeyen's standard references on this subject (\cite{nvo1}, \cite{nvo2}) cover a lot of properties of (the categories of) such graded structures. The aim of this note is to complement and extend several results found therein.

Before explaining the results in detail we recall the basic definitions as well as notation and terminology used throughout. In general, we follow Bourbaki's \textit{\'El\'ements de math\'ematique.} Monoids and groups are understood to be additively written and commutative, and rings are understood to be commutative.\smallskip

\textbf{Graded rings and modules.} Let $G$ be a group. A \textit{$G$-graded ring} is a pair $(R,(R_g)_{g\in G})$ consisting of a ring $R$ and a family $(R_g)_{g\in G}$ of subgroups of the additive group of $R$ whose direct sum equals the additive group of $R$ such that $R_gR_h\subseteq R_{g+h}$ for all $g,h\in G$. Usually we denote a $G$-graded ring $(R,(R_g)_{g\in G})$ just by $R$. If $R$ and $S$ are $G$-graded rings, then a \textit{morphism of $G$-graded rings from $R$ to $S$} is a morphism of rings $u\colon R\rightarrow S$ such that $u(R_g)\subseteq S_g$ for $g\in G$. We denote by $\ann^G$ the category of $G$-graded rings with the above notion of morphisms.

Let $R$ be a $G$-graded ring. A \textit{$G$-graded $R$-module} is a pair $(M,(M_g)_{g\in G})$ consisting of an $R$-module $M$ and a family $(M_g)_{g\in G}$ of subgroups of the additive group of $M$ whose direct sum equals the additive group of $M$ such that $R_gM_h\subseteq M_{g+h}$ for all $g,h\in G$. Usually we denote a $G$-graded $R$-module $(M,(M_g)_{g\in G})$ just by $M$. If $M$ and $N$ are $G$-graded $R$-modules, then a \textit{morphism of $G$-graded $R$-modules from $M$ to $N$} is a morphism of $R$-modules $u\colon M\rightarrow N$ such that $u(M_g)\subseteq N_g$ for $g\in G$. We denote by $\md(R)$ the category of $G$-graded $R$-modules with the above notion of morphisms. Furthermore, for a $G$-graded $R$-module $M$ we denote by $M\hme\dfgl\bigcup_{g\in G}M_g$ the set of homogeneous elements of $M$. An element $x\in R\hme$ is called \textit{invertible} or a \textit{unit} if there exists $y\in R\hme$ (or equivalently, $y\in R$) such that $xy=1$. We denote by $R^*$ the multiplicative group of invertible homogeneous elements of $R$. The set $\degsupp(R)\dfgl\{g\in G\mid R_g\neq 0\}$ is called \textit{the degree support of $R$.} The $G$-graded ring $R$ is called \textit{trivially $G$-graded} if $\degsupp(R)\subseteq 0$, i.e., if $R_g=0$ for every $g\in G\setminus 0$; this holds if and only if every element of $R$ is homogeneous. Further notation and terminology for graded rings and modules follow \cite{cihf} and \cite{gic}.

It might be tempting -- and the author admits having yielded to this temptation in previous work -- to denote the category of $G$-graded $R$-modules by $\md^G(R)$ or some such symbol containing the letter $G$. However, from the point of view of our yoga of coarsening this would be bad, or at least unnecessary, since the group $G$ is \textit{inherent to $R$,} which is not just a ring, but a \textit{$G$-graded ring.} As we throughout stress coarsening functors, and in particular forgetful functors, our choice of notation should not give rise to confusion. Similarly, we will denote $G$-graded Hom modules by $\hm{R}{M}{N}$ instead of, e.g., ${}^G\hm{R}{M}{N}$ (cf. \ref{dim70}).\smallskip

\textbf{Coarsening functors.} Let $\psi\colon G\twoheadrightarrow H$ be an epimorphism of groups, and let $R$ be a $G$-graded ring. The \textit{$\psi$-coarsening $R\ps$ of $R$} is the $H$-graded ring whose underlying ring is the ring underlying $R$ and whose component of degree $h\in H$ is $\bigoplus_{g\in\psi^{-1}(h)}R_g$. An analogous construction for graded modules yields the \textit{$\psi$-coarsening functor} $\bullet\ps\colon\md(R)\rightarrow\md(R\ps)$, coinciding for $H=0$ with the functor that forgets the graduation. Coarsening functors allow us to compare categories of graded rings or modules with different groups of degrees and are therefore an important tool in the study of graded structures. We refer the reader to \cite[Section 1]{cihf} for some generalities on coarsening functors.\smallskip

\textbf{A remark on terminology.} When choosing names for properties of graded objects, it seems preferable to use adjectives instead of nouns. As an example we consider graded rings whose nonzero homogeneous elements are invertible. In the ungraded situation, such a ring is called a \textit{field.} But a \textit{graded field} may be a graded ring with the aforementioned property, or a field (in the ungraded sense) furnished with a graduation. To avoid this ambiguity (and borrowing terminology from non-commutative algebra) we will call such a graded ring a \textit{simple graded ring.} Similarly, we will speak of \textit{entire graded rings} instead of \textit{graded domains} (cf. \ref{ser10}). For aesthetic reasons we will avoid names like \textit{gr-field} or \textit{gr-domain} as in \cite{nvo1} and \cite{nvo2}.\smallskip

\textbf{Overview of the results.} 1.\ \textit{Adjoints of degree restriction.} Let $\ph\colon F\rightarrowtail G$ be a monomorphism of groups. For a $G$-graded ring $R$ we consider the \textit{$\ph$-restriction $R_{(\ph)}$ of $R$,} i.e., the $F$-graded ring with underlying ring $\bigoplus_{f\in F}R_{\ph(f)}$ whose component of degree $f\in F$ is $R_{\ph(f)}$. This can be extended to a functor $\bullet_{(\ph)}\colon\ann^G\rightarrow\ann^F$, coinciding for $F=0$ with the functor that takes the component of degree $0$. It is well-known that $\bullet_{(\ph)}$ is right adjoint to the \textit{$\ph$-extension functor} $\bullet^{(\ph)}$ that maps an $F$-graded ring $R$ to the $G$-graded ring $R^{(\ph)}$ with underlying ring the ring underlying $R$ whose component of degree $g\in G$ is $R_f$ if $f\in\ph^{-1}(g)$, and $0$ otherwise (\cite[1.2.1]{nvo2}). We extend this to an adjoint triple $(\bullet\lph,\bullet^{(\ph)},\bullet_{(\ph)})$ and show that if $\ph$ is not an isomorphism, then $\bullet\lph$ has no left adjoint and $\bullet_{(\ph)}$ has no right adjoint (\ref{adj30}, \ref{adj50}). Adjoints of degree restriction functors for graded modules were treated by Menini and N\u{a}st\u{a}sescu (\cite{mn}, \cite{nas}).\smallskip

2.\ \textit{Simplicity, entirety, and reducedness.}
A graded ring is \textit{simple,} \textit{entire,} or \textit{reduced} if all its nonzero homogeneous elements are invertible, regular, or not nilpotent. We study the behaviour under coarsening functors of simplicity and reducedness, while entirety was already treated in \cite{gic}. It is known that no non-trivial coarsening respects simplicity (\cite[2.12]{gic}). Conversely, we show that if $\ke(\psi)$ is torsionfree and $R\ps$ is simple, then $R$ and $R\ps$ have the same homogeneous elements (\ref{simp10}). As a special case we get back the known result that a field can only be trivially graded by a torsionfree group. Special cases of this, that inspired our proof, were considered in \cite[A.I.1.2.5]{nvo1} and \cite[1.3.10]{nvo2}. Further results, mostly in a non-commutative setting, were proven by Jespers in \cite{jespers}. Moreover, we prove that $\ke(\psi)$ is torsionfree if and only if $\psi$-coarsening respects reducedness or, equivalently, turns simple graded rings into reduced graded rings (\ref{ser110}); this was previously stated without proof in \cite[2.13]{gic}. As an application we present several results on entirety, reducedness and units of algebras of monoids graded in different ways (\ref{ser120}, \ref{ser135}).\smallskip

3.\ \textit{Free graded modules and principal graded rings} A graded module is \textit{free} if it has a basis consisting of homogeneous elements. Coarsening functors respect freeness, but do not reflect it; an example is given in \cite[A.I.2.6.2]{nvo1}. Inspired by this example, we show that no non-trivial coarsening reflects freeness (\ref{a20}). On the positive side, we prove that over a \textit{principal graded ring,} i.e., an entire graded ring all of whose graded ideals are generated by a single homogeneous element, coarsening functors respect and reflect freeness of graded modules (\ref{a90}). We will not discuss the rather delicate behaviour of principality under coarsening, but hope for its pursuit in future work.\smallskip

4.\ \textit{Superfluous monomorphisms.} It is well-known that essential monomorphisms of graded modules are respected and reflected by coarsening functors (\cite[A.I.2.8]{nvo1}). The dual notion is that of superfluous monomorphisms, and it is easy to see that superfluous monomorphisms are reflected by coarsening functors. Conversely, we show that no non-trivial coarsening respects superfluous monomorphisms (\ref{super10}). Our proof is inspired by an example of such behaviour given in \cite[A.I.2.9]{nvo1}.\smallskip

5.\ \textit{Homological dimensions.} Categories of graded modules are abelian with\linebreak enough projectives and injectives, hence we can define the projective dimension $\dimp(M)$ and the injective dimension $\dimi(M)$ of a graded module $M$. Moreover, there is a notion of flat graded module, and as flatness implies projectivity we can define the flat dimension $\dimpl(M)$ of $M$. In \cite[3.1, 3.2]{fofo} and \cite[A.I.2.7, A.I.2.19]{nvo1} it is stated, essentially without proof, that coarsening functors preserve $\dimp(M)$ and $\dimpl(M)$; a more explicit proof for $\dimp(M)$ was given by Rigal and Zadunaisky (\cite{rz}). We provide a complete proof of Schanuel's Lemma for abelian categories (\ref{dim10}) and derive some consequences on preservation of projective dimension by exact functors; the original result for modules over not necessarily commutative rings is given in \cite[Theorem III.1]{kaplansky}. Although these results are probably well-known, their proofs in the desired generality seem to be not readily available in the literature. Furthermore, we prove a graded variant of Lambek's Lemma that may be of interest on its own (\ref{dim90}); the ungraded original can be found in \cite{lambek}. Then, we apply these results to obtain a complete proof of preservation of projective and flat dimensions by coarsening functors, and using a previous result on coarsening of injectives we show that $\bullet\ps$ preserves $\dimi(M)$ if and only if $\ke(\psi)$ is finite (\ref{dim120}). For more specific results on the behaviour of the injective dimension under coarsening functors we refer the reader to the recent work of Rigal, Solotar and Zadunaisky (\cite{rz}, \cite{sz}).\smallskip

We present graded variants of several well-known facts with complete proofs; it should be noted that although these proofs are often similar or even the same as (or rather graded variants of) the ungraded ones, the graded results do not just follow from the ungraded ones. The reason is usually that some of the involved notions do not behave well under coarsening.\smallskip

\noindent\textit{Throughout the following, let $\psi\colon G\twoheadrightarrow H$ be an epimorphism of groups, and let $R$ be a $G$-graded ring.}


\section{Adjoints of degree restriction}

\noindent\textit{Throughout this section, let $\ph\colon F\rightarrowtail G$ be a monomorphism of groups.}

\begin{no}
A) For a $G$-graded ring $R$ we define an $F$-graded ring $R_{(\ph)}$ whose underlying ring is the subring $\bigoplus_{f\in F}R_{\ph(f)}$ of the ring underlying $R$ and whose $F$-graduation is $(R_{\ph(f)})_{f\in F}$. A morphism of $G$-graded rings $h\colon R\rightarrow S$ induces by restriction and coastriction a morphism of $F$-graded rings $h_{(\ph)}\colon R_{(\ph)}\rightarrow S_{(\ph)}$. This gives rise to a functor $\bullet_{(\ph)}\colon\ann^G\rightarrow\ann^F$, called \textit{$\ph$-restriction.}\smallskip

B) For an $F$-graded ring $R$ we define a $G$-graded ring $R^{(\ph)}$ whose underlying ring is the ring underlying $R$ and whose $G$-graduation is given by $(R^{(\ph)})_g=0$ for $g\in G\setminus\im(\ph)$ and $(R^{(\ph)})_g=R_f$ for $g=\ph(f)\in\im(\ph)$. A morphism of $F$-graded rings $h\colon R\rightarrow S$ can be considered as a morphism of $G$-graded rings $h^{(\ph)}\colon R^{(\ph)}\rightarrow S^{(\ph)}$. This gives rise to a functor $\bullet^{(\ph)}\colon\ann^F\rightarrow\ann^G$, called \textit{$\ph$-extension.}  If $F=0$, then we write $\bullet^{(G)}$ instead of $\bullet^{(\ph)}$.\smallskip

C) For a $G$-graded ring $R$ we consider the graded ideal $\ia_\ph(R)\dfgl\langle\bigcup_{g\in G\setminus\im(\ph)}R_g\rangle_R$ of $R$ and the $F$-graded ring $R\lph=(R/\ia_\ph(R))_{(\ph)}$. The underlying ring of $R\lph$ being the ring underlying $R/\ia_\ph(R)$, a morphism of $G$-graded rings $h\colon R\rightarrow S$ induces by corestriction and astriction a morphism of $F$-graded rings $h\lph\colon R\lph\rightarrow S\lph$. This gives rise to a functor $\bullet\lph\colon\ann^G\rightarrow\ann^F$, called \textit{$\ph$-corestriction.}\smallskip

D) If $S$ is an $F$-graded ring, then $(S^{(\ph)})_{(\ph)}=S=(S^{(\ph)})\lph$. Let $R$ be a $G$-graded ring. Then, $(R_{(\ph)})^{(\ph)}$ is a $G$-graded subring of $R$, so we have a monomorphism of $G$-graded rings $(R_{(\ph)})^{(\ph)}\hookrightarrow R$ that is natural in $R$. Furthermore, for $g\in G\setminus\im(\ph)$ we have $((R\lph)^{(\ph)})_g=0$, so the zero morphism is an epimorphism of $R_0$-modules $\alpha_\ph(R)_g\colon R_g\twoheadrightarrow((R\lph)^{(\ph)})_g$. For $g\in\im(\ph)$ we have $((R\lph)^{(\ph)})_g=R_g/\sum_{h\in G\setminus\im(\ph)}R_hR_{g-h}$, so the canonical projection is an epimorphism of $R_0$-modules $\alpha_\ph(R)_g\colon R_g\twoheadrightarrow((R\lph)^{(\ph)})_g$. Thus, $(\alpha_\ph(R)_g)_{g\in G}$ defines a surjective morphism of $G$-graded rings $\alpha_\ph(R)\colon R\twoheadrightarrow(R\lph)^{(\ph)}$ that is natural in $R$.
\end{no}

\begin{prop}\label{adj10}
a) If there exists $x\in R^*$ with $\deg(x)\notin\im(\ph)$, then $R\lph=0$.

b) If $(\degsupp(R)\setminus\im(\ph)+\degsupp(R)\setminus\im(\ph))\cap\degsupp(R)\cap\im(\ph)=\emptyset$, then $R\lph=R_{(\ph)}$; the converse holds if $R$ is entire.
\end{prop}

\begin{proof}
a) If $g\in G\setminus\im(\ph)$, $x\in R_g$ and $y\in R_{-g}$ with $xy=1$, then $1\in\ia_\ph(R)$, hence $R\lph=0$.

b) If $R\lph\neq R_{(\ph)}$, then there exist $g,h\in G\setminus\im(\ph)$, $x\in R_g$ and $y\in R_h$ with $g+h\in\im(\ph)$ and $xy\neq 0$, hence with $x\neq 0$ and $y\neq 0$, implying that $g,h\in\degsupp(R)$. Conversely, suppose that $R$ is entire. If there exist $g,h\in\degsupp(R)\setminus\im(\ph)$ with $g+h\in\im(\ph)$, then there exist $x\in R_g\setminus 0$ and $y\in R_h\setminus 0$. Entirety implies $0\neq xy\in\ia_\ph(R)_{g+h}$, and thus $(R\lph)_{g+h}\neq R_{g+h}=(R_{(\ph)})_{g+h}$.
\end{proof}

\begin{cor}
If $\degsupp(R)\cap(-\degsupp(R))=0$, then $R_{(\!(0)\!)}=R_{(0)}$; the converse holds if $R$ is entire.
\end{cor}

\begin{proof}
Immediately from \ref{adj10} b) with $F=0$.
\end{proof}

\begin{exas}\label{adj20}
A) Let $A$ be an $F$-graded ring, let $(g_i)_{i\in I}$ be a family in $G$ such that there exists $i\in I$ with $g_i\notin\im(\ph)$, and let $R=A^{(\ph)}[(X_i)_{i\in I},(X_i^{-1})_{i\in I}]$ be the Laurent algebra over $A$ in the variables $(X_i)_{i\in I}$, furnished with the $G$-graduation given by $\deg(X_i)=g_i$ for $i\in I$. Then, $R\lph=0$.\smallskip

B) If $R$ is a positively $\Z$-graded ring, then $R_{(\!(0)\!)}=R_{(0)}$.\smallskip

C) Let $g\in G$, let $n$ be the order of the class of $g$ in $G/\im(\ph)$, and let $A$ be an $F$-graded ring. Let $R=A^{(\ph)}[X]$ be the polynomial algebra over $A^{(\ph)}$ in the variable $X$ in case $n=\infty$, and let $R=A^{(\ph)}[X]/\langle X^n\rangle$ be the quotient thereof by $\langle X^n\rangle$ in case $n<\infty$, furnished with the $G$-graduation such that the degree of (the class of) $X$ is $g$. Then, $R\lph=R_{(\ph)}$.\smallskip

D) If $R$ is an entire $G$-graded ring with $\im(\ph)\neq\degsupp(R)$ and $G/\im(\ph)$ is finite, then $R\lph\neq R_{(\ph)}$.
\end{exas}

\begin{prop}\label{adj30}
There is an adjoint triple $(\bullet\lph,\bullet^{(\ph)},\bullet_{(\ph)})$ with units\linebreak $\alpha_\ph\colon\Id_{\ann^G}(\bullet)\twoheadrightarrow(\bullet\lph)^{(\ph)}$ and\/ $\Id\colon\Id_{\ann^F}\rightarrow(\bullet^{(\ph)})_{(\ph)}$, and with counits\linebreak $\Id\colon(\bullet^{(\ph)})\lph\rightarrow\Id_{\ann^F}$ and $(\bullet_{(\ph)})^{(\ph)}\hookrightarrow\Id_{\ann^G}(\bullet)$.
\end{prop}

\begin{proof}
Straightforward.
\end{proof}

\begin{no}\label{adj40}
If $(F,G,H)$ is an adjoint triple, then the following statements are equivalent (\cite[1.5.6, Exercise 1.14]{ks}): (i) $F$ is fully faithful; (ii) the unit of $(F,G)$ is an isomorphism; (iii) $H$ is fully faithful; (iv) the counit of $(G,H)$ is an isomorphism. Moreover, if they hold then $G$ is conservative.
\end{no}

\begin{prop}\label{adj50}
If $\ph$ is not an isomorphism, then $\bullet\lph$ has no left adjoint and $\bullet_{(\ph)}$ has no right adjoint.
\end{prop}

\begin{proof}
If $\ph$ is not an isomorphism, then there exists $g\in G\setminus\im(\ph)$. So, $\bullet\lph$ maps non-isomorphic $G$-graded rings to the zero ring (\ref{adj20} A)) and thus is not conservative. Moreover, the counit of $(\bullet\lph,\bullet^{(\ph)})$ is an isomorphism (\ref{adj30}), hence $\bullet\lph$ has no left adjoint (\ref{adj40}). Furthermore, identifying $R\otimes_{\Z}R$ with $S$ we get $XY^{-1}\in R_g\otimes_{\Z}R_{-g}\subseteq(R\otimes_{\Z}R)_{(\ph)}$, and as $R_{(\ph)}=\Z^{(F)}$ we get $XY^{-1}\notin\Z^{(F)}=R_{(\ph)}\otimes_{\Z}R_{(\ph)}$. So, $\bullet_{(\ph)}$ does not commute with tensor products over $\Z$. As these are coproducts in $\ann^G$ and $\ann^F$ it follows that $\bullet\lph$ has no right adjoint.
\end{proof}


\section{Simplicity, entirety, and reducedness}

\begin{no}\label{ser10}
A) Let $g\in G$, and let $M$ be a $G$-graded $R$-module. We denote by $M(g)$ the $g$-shift of $M$, i.e., the $G$-graded $R$-module whose underlying $R_{[0]}$-module is the $R_{[0]}$-module underlying $M$ and whose $G$-graduation is given by $M(g)_h=M_{g+h}$ for $h\in G$.\smallskip

B) If $g\in G$ and $x\in R_g$, then multiplication by $x$ defines a morphism of $G$-graded $R$-modules $m_x\colon R\rightarrow R(g),\;y\mapsto xy$. (As we are only interested in $m_x$ being an iso- or a monomorphism, no problems will arise from the ambiguity of $m_0$.) An element $x\in R\hme$ is invertible if and only if $m_x$ is an isomorphism. The $G$-graded ring $R$ is called \textit{simple} if $R^*=R\hme\setminus 0$, i.e., if $R\neq 0$ and every nonzero homogeneous element is invertible.\smallskip

C) An element $x\in R\hme$ is called \textit{regular} or a \textit{non-zerodivisor} if $m_x$ is a monomorphism, and a \textit{zerodivisor (of $R$)} otherwise. The latter holds if and only if there exists $y\in R\hme\setminus 0$ with $xy=0$. We denote by $\nzd(R)$ the multiplicative monoid of regular homogeneous elements of $R$ and by $\zd(R)$ the graded ideal of $R$ generated by all homogeneous zerodivisors of $R$. The $G$-graded ring $R$ is called \textit{entire} if $\nzd(R)=R\hme\setminus 0$, i.e., if $R\neq 0$ and every nonzero homogeneous element is regular. This holds if and only if $R\neq 0$ and $\zd(R)=0$.\smallskip

D) An element $x\in R\hme$ is called \textit{nilpotent} if there exists $p\in\N$ with $x^p=0$. We denote by $\nil(R)$ the graded ideal of $R$ generated by all nilpotent homogeneous elements of $R$. The $G$-graded ring $R$ is called \textit{reduced} if $\nil(R)=0$, i.e., if no nonzero homogeneous element is nilpotent. Clearly, $R$ is reduced if and only if $x^2\neq 0$ for every $x\in R\hme\setminus 0$.\smallskip

E) We have $\nil(R)\subseteq\zd(R)$ and $R^*\subseteq\nzd(R)\subseteq R\hme$. Hence, a simple $G$-graded ring is entire, and an entire $G$-graded ring is reduced. Furthermore, the $G$-graded zero ring is reduced, but not entire.
\end{no}

\begin{no}
A) If $\ia\subseteq R$ is a graded ideal and $\pi\colon R\twoheadrightarrow R/\ia$ is the canonical projection, then the graded ideal $\sqrt\ia\dfgl\pi^{-1}(\nil(R/\ia))=\langle x\in R\hme\mid\exists n\in\N\colon x^n\in\ia\rangle_R$ of $R$ is called \textit{the radical of $\ia$.} Clearly, $\nil(R)=\sqrt0$.\smallskip

B) A graded ideal $\ia\subseteq R$ is called \textit{maximal,} \textit{prime,} or \textit{perfect,} if the $G$-graded ring $R/\ia$ is simple, entire, or reduced, resp. So, $\ia$ is maximal if and only if it is $\subseteq$-maximal among all proper graded ideals of $R$, prime if and only if $R\hme\setminus\ip$ is multiplicatively closed, and perfect if and only if $\sqrt\ia=\ia$.\smallskip

C) The set of all prime graded ideals of $R$ is denoted by $\spec(R)$ and called \textit{the spectrum of $R$.} The set of all prime graded ideals of $R$ containing a graded ideal $\ia\subseteq R$ is denoted by $\var(\ia)$ and called \textit{the variety of $\ia$.}
\end{no}

\begin{prop}\label{ser60}
Let $S\subseteq R\hme$, and let $\ia\subseteq R$ be a graded ideal with $\ia\cap S=\emptyset$. The set of graded ideals of $R$ containing $\ia$ and not meeting $S$ has a maximal element; if $S$ is multiplicatively closed, then every such maximal element is prime.
\end{prop}

\begin{proof}
The set $\mathcal{I}$ of graded ideals of $R$ containing $\ia$ and not meeting $S$, ordered by $\subseteq$, is nonempty and inductive, hence has a maximal element $\ip$ by Zorn's Lemma. Suppose that $S$ is multiplicatively closed. Then, $S\neq\emptyset$ and hence $\ip\neq R$. Let $x,y\in R\hme\setminus\ip$. Then, $\ip+\langle x\rangle_R,\ip+\langle y\rangle_R\notin\mathcal{I}$, but both these graded ideals contain $\ia$. So, there exist $u,v\in\ip\hme$ and $a,b\in R\hme$ with $u+ax,v+by\in S$. This implies that $uv+uyb+vxa+xyab=(u+ax)(v+by)\in S\subseteq R\hme\setminus\ip$. As $uv+uyb+vxa\in\ip$ it follows that $xyab\notin\ip$, hence $xy\notin\ip$. Thus, $\ip$ is prime.
\end{proof}

\begin{cor}\label{ser70}
a) If $\ia\subseteq R$ is a graded ideal, then $\sqrt\ia=\bigcap\var(\ia)$.

b) $\nil(R)=\bigcap\spec(R)$.
\end{cor}

\begin{proof}
a) If $x\in\sqrt\ia\hme$ and $\ip\in\var(\ia)$, then there exists $p\in\N$ with $x^p\in\ia\subseteq\ip$, hence $x\in\ip$. This shows that $\sqrt\ia\subseteq\bigcap\var(\ia)$. Conversely, if $x\in R\hme\setminus\sqrt\ia$, then $S\dfgl\{x^p\mid p\in\N\}\subseteq R$ is multiplicatively closed, hence by \ref{ser60} there exists $\ip\in\var(\ia)$ with $\ip\cap S=\emptyset$, implying $x\notin\ip$ and therefore $x\notin\bigcap\var(\ia)$. This shows that $\bigcap\var(\ia)\subseteq\sqrt\ia$. b) Apply a) with $\ia=0$.
\end{proof}

\begin{no}\label{ser99}
A) If $x\in R\hme$, then its property of being invertible, regular, or nilpotent depends only on $x$ and the underlying ring of $R$, but not on the graduation of $R$. So, as $R\hme\subseteq R\ps\hme$, we have $R^*\subseteq R\ps^*$, $\nzd(R)\subseteq\nzd(R\ps)$, $\zd(R)\ps\subseteq\zd(R\ps)$, and $\nil(R)\ps\subseteq\nil(R\ps)$. If $R$ is entire and $\ke(\psi)$ is torsionfree, then $R^*=R\ps^*$ by \cite[2.12 a)]{gic}.\smallskip

B) If $R\ps$ is simple, entire, or reduced, then so is $R$. Concerning the converse, we know from \cite[2.10, 2.12]{gic} that the functor $\bullet\ps$ preserves simplicity or entirety if and only if $\psi$ is an isomorphism or has a torsionfree kernel, resp. (cf. \ref{ser110}).\smallskip

C) Let $\ia\subseteq R$ be a graded ideal. Then, $\sqrt\ia\ps\subseteq\sqrt{\ia\ps}$. Moreover, if $\ia\ps$ is maximal, prime, or perfect, then so is $\ia$.
\end{no}

\begin{prop}\label{simp10}
If $R$ is simple and $R\hme=R\ps\hme$, then $R\ps$ is simple. The converse holds if\/ $\ke(\psi)$ is torsionfree.
\end{prop}

\begin{proof}
The first claim is clear. Suppose that $\ke(\psi)$ is torsionfree and that $R\ps$ is simple. Then, $R$ is simple and $\ke(\psi)$ can be furnished with a structure of totally ordered group (\cite[II.11.4 Lemme 1]{a}). Let $g\in\ke(\psi)\cap\degsupp(R)$. There exists $x\in R_g\setminus 0$. As $x^{-1}\in R_{-g}\setminus 0$ we may replace $x$ by $x^{-1}$ and thus suppose without loss of generality that $g\geq 0$. Now, $1+x\in(R\ps)_{0}$ is invertible, so there exists $y\in R$ with $1=(1+x)y$. Hence, there exist $s\in\N$, a strictly increasing sequence $(g_i)_{i=0}^s$ in $\ke(\psi)$, and $(y_{g_i})_{i=0}^s\in\prod_{i=0}^s(R_{g_i}\setminus 0)$ with $y=\sum_{i=0}^sy_{g_i}$. It follows that $1=(1+x)y=\sum_{i=0}^sy_{g_i}+\sum_{i=0}^sxy_{g_i}$ is homogeneous. On the right side, the component of smallest degree is $y_{g_0}\in R_{g_0}$, and the component of largest degree is $xy_{g_s}\in R_{g+g_s}$. This implies that $g_0=g+g_s$, hence $s=0$ and $g=0$, and therefore $\ke(\psi)\cap\degsupp(R)=0$. Next, let $g,g'\in G\setminus\ke(\psi)$ with $\psi(g)=\psi(g')$. There exist $x\in R_g\setminus 0$ and $y\in R_{g'}\setminus 0$. It follows that $y^{-1}\in R_{-g'}\setminus 0$, hence $xy^{-1}\in R_{g-g'}\setminus 0$, thus $g-g'\in\ke(\psi)\cap\degsupp(R)$, and therefore $g=g'$. So, for every $h\in H$ the cardinality of $\psi^{-1}(h)\cap\degsupp(R)$ is at most $1$, implying that $R\hme=R\ps\hme$.
\end{proof}

\begin{cor}
If $G$ is a torsionfree group and $R_{[0]}$ is a field, then $R$ is trivially $G$-graded.
\end{cor}

\begin{proof}
Immediately from \ref{simp10} with $H=0$.
\end{proof}

\begin{no}\label{ser100}
A) If $R\ps$ is simple and $\ke(\psi)$ is not torsionfree, then we may have $R\hme\subsetneqq R\ps\hme$, as illustrated by the following examples (taken from \cite[1.3.7--8]{nvo2}).\smallskip

B) Let $n\in\N_{>1}$, and let $K\subseteq L$ be a field extension such that there exists $x\in L$ algebraic over $K$ with $L=K(x)$ whose minimal polynomial has the form $X^n-a$ for some $a\in K$. Then, $(\langle x^\alpha\rangle_K)_{\alpha\in\Z/n\Z}$ is a $\Z/n\Z$-graduation on $L$. We furnish $L$ with this $\Z/n\Z$-graduation and denote the $\Z/n\Z$-graded ring thus obtained by $R$. Then, $R_{[0]}=L$ is simple, hence so is $R$. However, $R\hme\subsetneqq R_{[0]}\hme=L$, since, e.g., $x+1\notin R\hme$.\smallskip

C) Let $n\in\N_{>1}$, and let $K$ be a field. Applying B) with $x=X$ to the field extension $K(X^n)\subseteq K(X)$, where $K(X)$ is the field of rational fractions over $K$ in one indeterminate $X$, we get a $\Z/n\Z$-graded ring $R$ with $R_{[0]}=K(X)$ and $R\hme\subsetneqq R_{[0]}\hme$.\smallskip

D) Applying B) with $x=i$ to the field extension $\R\subseteq\mathbbm{C}$ we get a $\Z/2\Z$-graded ring $R$ with $R_{[0]}=\mathbbm{C}$ such that $R_{\overline{0}}=\R$ and $R_{\overline{1}}=\R i$. Clearly, $R\hme$ consists of the real and the purely imaginary numbers.
\end{no}

\begin{no}
A) For a monoid $M$ we denote by $\diff(M)$ and $M^*$ the groups of differences and of invertible elements of $M$. The monoid $M$ is cancellable if and only if the canonical morphism of monoids $M\rightarrow\diff(M)$ is a monomorphism. A cancellable monoid $M$ is torsionfree if and only if the group $\diff(M)$ is so. Furthermore, $M$ is called \textit{sharp} if $M^*=0$.\smallskip

B) Let $M$ be a cancellable monoid. Following \cite[2.5]{gic}, we denote by $R[M]$ \textit{the finely graded algebra of $M$ over $R$,} i.e., the algebra of $M$ over $R_{[0]}$, furnished with its canonical $G\oplus\diff(M)$-graduation. Denoting by $(e_m)_{m\in M}$ its canonical basis, we have $\deg(e_m)=(0,m)$ for $m\in M$. Furthermore, we denote by $R[M]_{[G]}$ \textit{the coarsely graded algebra of $M$ over $R$,} i.e., the $G$-graded $R$-algebra $R[M]_{[\pi]}$ where $\pi\colon G\oplus\diff(M)\twoheadrightarrow G$ is the canonical projection. Clearly, we have $\deg(e_m)=0$ for $m\in M$. More generally, if $d\colon M\rightarrow G$ is a morphism of monoids, then we denote by $R[M;d]$ \textit{the $d$-graded algebra of $M$ over $R$,} i.e., the $G$-graded $R$-algebra $R[M]_{[\delta]}$ where $\delta\colon G\oplus\diff(M)\twoheadrightarrow G$ is the epimorphism of groups induced by $d$ whose restriction to $G$ is $\Id_G$. Clearly, we have $\deg(e_m)=d(m)$ for $m\in M$.
\end{no}

\begin{lemma}\label{ser101}
Suppose that $\ke(\psi)$ is torsionfree.

a) Let $x,y\in R\ps\hme\setminus 0$ with $xy\in R\hme$. If $R$ is entire, then $x,y\in R\hme$ and $xy\neq 0$.

b) Let $x\in R\ps\hme\setminus 0$ and $p\in\N$ with $x^p\in R\hme$. If $R$ is reduced, then $x\in R\hme$ and $x^p\neq 0$.
\end{lemma}

\begin{proof}
a) holds by \cite[2.9]{gic}. b) By \cite[II.1.4 Lemme 1]{a} we can choose a total ordering on $\ke(\psi)$ that is compatible with its structure of group. Let $\leq$ denote its canonical extension to $G$ (\cite[2.8]{gic}). Let $h\dfgl\deg(x)\in H$. There exist a strictly increasing sequence $(g_i)_{i=0}^n$ in $\psi^{-1}(h)$ and $(x_i)_{i=0}^n\in\prod_{i=0}^n(R_{g_i}\setminus 0)$ such that $x=\sum_{i=0}^nx_i$. If $(k_j)_{j=1}^p$ is a sequence in $[0,n]$ with $\sum_{j=1}^pg_{k_j}=pg_n$, then $k_j=n$ for every $j\in[0,p]$ (\cite[VI.1.1 Proposition 1]{a}), hence the component of $x^p$ of degree $pg_n$ equals $x_n^p\neq 0$, and thus we have $x^p\neq 0$. As $x_0^p\neq 0$ and $x^p\in R\hme$ we have $pg_0=pg_n$. As we saw above this implies $n=0$ and thus the claim.
\end{proof}

\begin{thm}\label{ser110}
The following statements are equivalent: (i) $\ke(\psi)$ is torsionfree; (ii) $\bullet\ps$ respects entirety; (iii) $\bullet\ps$ maps simple $G$-graded rings to entire $H$-graded rings; (iv) $\bullet\ps$ respects reducedness; (v) $\bullet\ps$ maps simple $G$-graded rings to reduced $H$-graded rings.
\end{thm}

\begin{proof}
``(i)$\Leftrightarrow$(ii)$\Leftrightarrow$(iii)'' holds by \cite[2.12 b)]{gic}\footnote{The remaining claims were mentioned without proof in \loccit}. ``(i)$\Rightarrow$(iv)'' holds by \ref{ser101} b).\linebreak ``(iv)$\Rightarrow$(v)'' is clear. ``(v)$\Rightarrow$(i)'': Suppose that (v) holds, let $p$ be a prime number, and let $K$ be a field of characteristic $p$. Then, the $G$-graded ring $K[\ke(\psi)]^{(G)}$ is simple, hence the (ungraded) ring $K[\ke(\psi)]_{[0]}$ is reduced by (v), and thus $\ke(\psi)$ is $p$-torsionfree by \cite[9.3]{gilmer} (which says that given a prime number $p$, an entire ring $A$ of characteristic $p$ and a group $M$, the (ungraded) ring $A[M]$ is reduced if and only if $M$ is $p$-torsionfree). As this holds for every prime number $p$, it follows that $\ke(\psi)$ is torsionfree.
\end{proof}

\begin{prop}\label{ser120}
a) Let $M$ be a cancellable monoid. Then, $R$ is entire (or reduced) if and only if $R[M]$ is so.

b) Let $M$ be a torsionfree, cancellable monoid, and let $d\colon M\rightarrow G$ be a morphism of monoids. Then, $R[M;d]$ is entire (or reduced) if and only if $R$ is so.
\end{prop}

\begin{proof}
a) If $R[M]$ is entire (or reduced), then so is its $G\oplus\diff(M)$-graded subring $R^{(G\oplus\diff(M))}$, and thus so is the $G$-graded ring $R$. Conversely, suppose first that $R$ is entire. Let $x,y\in R[M]\hme\setminus 0$. There exist $m,n\in M$ and $r,s\in R\hme\setminus 0$ with $x=re_m$ and $y=se_n$, so that $xy=rse_{m+n}$. As $R$ is entire, we have $rs\neq 0$, hence $xy\neq 0$. Therefore, $R[M]$ is entire. Suppose next that $R$ is reduced. Let $x\in R[M]\hme\setminus 0$ and $p\in\N$. There exist $m\in M$ and $r\in R\hme\setminus 0$ with $x=re_m$, so that $x^p=r^pe_{pm}$. As $R$ is reduced, we have $r^p\neq 0$, hence $x^p\neq 0$. Therefore, $R[M]$ is reduced.

b) If $R[M;d]$ is entire (or reduced), then so is its $G$-graded subring $R$. Conversely, suppose that $R$ is entire (or reduced). As $R[M;d]\hme\subseteq R[M]_{[G]}\hme$ we can without loss of generality suppose that $d=0$. So, it suffices to show that $R[M]_{[G]}$ is entire (or reduced). Denoting by $\delta\colon G\oplus\diff(M)\twoheadrightarrow G$ the canonical projection we have $R[M]_{[G]}=R[M]_{[\delta]}$. As $M$ is torsionfree, the same holds for $\ke(\delta)=\diff(M)$. Thus, \ref{ser110} and a) imply that $R[M]_{[G]}$ is entire (or reduced).
\end{proof}

\begin{prop}\label{ser135}
a) Let $M$ be a cancellable monoid. We consider the following statements: 
(i) $R[M]_{[G]}^*=R^*$; (ii) $R[M]^*=R^*$; (iii) $M$ is sharp. Then, we have (i)$\Rightarrow$(ii)$\Leftrightarrow$(iii), and if $R$ is entire and $M$ is torsionfree, then (i)--(iii) are equivalent.

b) Let $M$ be a sharp, torsionfree, cancellable monoid, let $d\colon M\rightarrow G$ be a morphism of monoids, and suppose that $R$ is reduced. Then, $R[M;d]^*=R^*$.
\end{prop}

\begin{proof}
a) We have $R^*\subseteq R[M]^*\subseteq R[M]_{[G]}^*$, and hence (i) implies (ii). If $M$ is sharp and $x\in R[M]^*$, then there exist $r,s\in R\hme$ and $m,n\in M$ with $x=re_m$ and $rse_me_n=1=e_0$, implying $rs=1$ and $m+n=0$, hence $r\in R^*$ and $m=0$, and thus $x\in R^*$. Conversely, if $R[M]^*=R^*$ and $m\in M^*$, then $e_me_{-m}=1$, hence $e_m\in R[M]^*=R^*$, and therefore $m=0$. This shows that (ii) and (iii) are equivalent. If $R$ is entire and $M$ is torsionfree, then $R[M]^*=R[M]_{[G]}^*$ by \ref{ser99} A), and thus the equivalence of (ii) and (iii) follows from the above.

b) We have $R^*\subseteq R[M;d]^*$. As $R[M;d]^*\subseteq R[M]_{[G]}^*$, we can without loss of generality suppose that $d=0$. So, it suffices to show $R[M]_{[G]}^*\subseteq R^*$. Let $x\in R[M]_{[G]}^*$. There exists a family $(r_m)_{m\in M}$ of finite support in $R\hme$ with $x=\sum_{m\in M}r_me_m$. Let $n\in M$ with $r_n\neq 0$. Since $R$ is reduced, \ref{ser70} b) implies that there exists $\ip\in\spec(R)$ with $r_n\notin\ip$. We consider the canonical projection $p\colon R[M]_{[G]}\rightarrow(R/\ip)[M]_{[G]}$. Since $R/\ip$ is entire and $M$ is torsionfree and sharp, we have $p(x)\in(R/\ip)[M]_{[G]}^*=(R/\ip)^*$ by a).  As $p(x)=\sum_{m\in M}(r_m+\ip)e_m$, this implies that $r_m\in\ip$ for every $m\neq 0$. Therefore, $n=0$, and thus $x=r_0e_0\in R^*$ as desired.
\end{proof}


\section{Free graded modules and principal graded rings}

\begin{no}\label{ser85}
A) Let $M$ be a $G$-graded $R$-module. A subset $E\subseteq M\hme$ is called \textit{free} if whenever $(r_e)_{e\in E}$ is a family of finite support in $R\hme$ with $\sum_{e\in E}r_ee=0$, then $r_e=0$ for every $e\in E$.\smallskip

B) Suppose that $R$ is simple. If $E\subseteq M\hme$ is free and $x\in M\hme\setminus\langle E\rangle_R$, then $E\cup\{x\}$ is free. Indeed, let $r\in R\hme$ and let $(r_e)_{e\in E}$ be a family of finite support in $R\hme$ with $rx+\sum_{e\in E}r_ee=0$. If $r\neq 0$, then we get the contradiction that $x=-\sum_{e\in E}\frac{r_e}{r}e\in\langle E\rangle_R$. So, $r=0$, hence $\sum_{e\in E}r_ee=0$, and therefore $r_e=0$ for every $e\in E$.\smallskip

C) A subset $E\subseteq M\hme$ is called a \textit{basis} of $M$ if it is free and generates $M$. This holds if and only if whenever $N$ is a $G$-graded $R$-module and $f\colon E\rightarrow N$ is a map with $\deg(f(e))=\deg(e)$ for every $e\in E$, then $f$ can be extended uniquely to a morphism of $G$-graded $R$-modules $M\rightarrow N$. The $G$-graded $R$-module $M$ is called \textit{free} if it has a basis. This holds if and only if there exists an isomorphism of $G$-graded $R$-modules $M\cong\bigoplus_{g\in G}R(g)^{\oplus E_g}$ for some family of sets $(E_g)_{g\in G}$.\smallskip

D) A basis of $M$ is a basis of $M\ps$. Therefore, if $M$ is free, then so is $M\ps$. From the ungraded case it thus follows that if $M$ is free, then all bases of $M$ have the same cardinality. This common cardinality is denoted by $\rk_R(M)$ and called \textit{the rank of $M$.} (If $R=0$ we convene that $\rk_R(M)=0$.) Clearly, if $M$ is free, then $\rk_R(M)=\rk_{R\ps}(M\ps)$. Finally, as in the ungraded case it is clear that free $G$-graded $R$-modules are projective.
\end{no}

\begin{no}
A $G$-graded $R$-module is called \textit{monogeneous} if it has a set of homogeneous generators of cardinality $1$. The $G$-graded ring $R$ is called \textit{principal} if it is entire and every graded ideal is monogeneous.
\end{no}

\begin{prop}\label{ser90}
Suppose that $R$ is simple, let $M$ be a $G$-graded $R$-module, and let $E\subseteq F\subseteq M\hme$ be subsets such that $F$ generates $M$ and that $E$ is free. Then, there exists a basis $B$ of $F$ with $E\subseteq B\subseteq F$.
\end{prop}

\begin{proof}
The set of free subsets of $M\hme$ containing $E$ and contained in $F$, ordered by $\subseteq$, is nonempty and inductive, hence has a maximal element $B$ by Zorn's Lemma. By \ref{ser85} B), $B$ generates $M$ and thus is the desired basis.
\end{proof}

\begin{cor}
If $R$ is simple, then every $G$-graded $R$-module is free.
\end{cor}

\begin{proof}
Apply \ref{ser90} with $E=\emptyset$ and $F=M\hme$.
\end{proof}

\begin{prop}\label{a20}
If $\psi$ is not an isomorphism, then there exist a $G$-graded ring $R$ and a $G$-graded $R$-module $M$ such that $M$ is not free but $M\ps$ is free of rank $1$.
\end{prop}

\begin{proof}
If $\psi$ is not an isomorphism, then there exists $f\in\ke(\psi)\setminus 0$. Let $R=(\Z\times\Z)^{(G)}$, and let $M$ be the $R$-module $\Z\times\Z$, furnished with the $G$-graduation given by $M_0=\Z\times 0$, $M_f=0\times\Z$, and $M_g=0$ for $g\in G\setminus\{0,f\}$. Then, $M\ps=R\ps$ is free of rank $1$. If $M$ is free, then $\rk_R(M)=1$ by \ref{ser85} D), but as $M$ is not monogeneous this is not possible. Thus, the claim is proven.
\end{proof}

\begin{prop}\label{a30}
Let $\P$ be a class of $G$-graded $R$-modules such that projective elements of $\P$ are free, and let $M\in\P$. Then, $M$ is free if and only if $M_{\ps}$ is so.
\end{prop}

\begin{proof}
If $M_{\ps}$ is free, then it is projective, hence $M$ is projective by \cite[A.I.2.2]{nvo1}, and thus $M$ is free. The converse holds by \ref{ser85} D).
\end{proof}

\begin{prop}\label{a50}
Let $L$ be a free $G$-graded $R$-module, let $M\subseteq L$ be a graded sub-$R$-module, and suppose that graded ideals of $R$ are projective. Then, there exist a family $(\ia_e)_{e\in E}$ of graded ideals of $R$ and a family $(g_e)_{e\in E}$ in $G$ such that $\card(E)=\rk_R(L)$ and $M\cong\bigoplus_{e\in E}\ia_e(g_e)$.
\end{prop}

\begin{proof}
Let $E$ be a homogeneous basis of $L$ and let $\leq$ be a well-ordering on $E$. For $e\in E$ we set $g_e\dfgl\deg(e)\in G$ and denote by $p_e\colon L\rightarrow R(g_e)$ the coordinate function of $e$ with respect to $E$ which is a morphism of $G$-graded $R$-modules. For $e\in E$ we consider the graded sub-$R$-module $L_e\dfgl\langle E_{\leq e}\rangle_R\subseteq L$, the graded sub-$R$-module $M_e\dfgl M\cap L_e\subseteq M\subseteq L$, the graded sub-$R$-module $p_e(M_e)\subseteq R(g_e)$, and the graded ideal $\ia_e\dfgl p_e(M_e)(-g_e)\subseteq R$. Then, $p_e$ induces by restriction and coastriction an epimorphism of $G$-graded $R$-modules $p_e'\colon M_e\twoheadrightarrow\ia_e(g_e)$. Since $\ia_e$ is projective, there exists a section $s_e\colon\ia_e(g_e)\rightarrow M_e$ of $p_e'$. We consider the graded sub-$R$-module $N_e\dfgl\im(s_e)\subseteq M_e$. As $N_e\cong\ia_e(g_e)$ for $e\in E$, it suffices to show that $M=\bigoplus_{e\in E}N_e$.

We show now that $M=\sum_{e\in E}N_e$. For $e\in E$ we consider the graded sub-$R$-module $M'_e\dfgl\sum_{f\in E_{\leq e}}N_f\subseteq M_e\subseteq M$. As $\sum_{e\in E}M'_e=\sum_{e\in E}N_e$, it suffices to show that $M=\sum_{e\in E}M'_e$. As $M=\sum_{e\in E}M_e$ it suffices to show that $M_e\subseteq M'_e$ for $e\in E$. This we do by induction on the well-ordered set $(E,\leq)$. Let $e\in E$, and suppose that $M_f\subseteq M'_f$ for every $f\in E_{<e}$. Let $x\in M_e$. Then, $p_e(x)\in\ia_e(g_e)$, hence there exists $y\in N_e$ with $x-y\in M\cap(\sum_{f<e}L_f)$, thus there exists $f\in E_{<e}$ with $x-y\in M\cap L_f=M_f\subseteq M'_f\subseteq M'_e$, and as $y\in M'_e$, we get $x\in M'_e$ as desired.

If the sum $\sum_{e\in E}N_e$ is not direct, then there exists $(a_e)_{e\in E}\in(\bigoplus_{e\in E}N_e)\setminus 0$ with $\sum_{e\in E}a_e=0$. Setting $f\dfgl\max\{e\in E\mid a_e\neq 0\}$ it follows $p_f(a_e)=0$ for $e\in E_{<f}$ and hence the contradiction $0\neq a_f=p_f(\sum_{e\in E}a_e)=p_f(0)=0$. Thus, the claim is proven.\footnote{This closely follows the proof of the ungraded variant in \cite[VII.3 Th\'eor\`eme 1]{a}, enhanced by some bookkeeping about degrees.}
\end{proof}

\begin{lemma}\label{a60}
Let $M$ be a $G$-graded $R$-module, let $L\subseteq M$ be a free graded sub-$R$-module, and let $n\in\N$. If $M$ has a homogeneous set of generators of cardinality $n$, then $\rk_R(L)\leq n$.
\end{lemma}

\begin{proof}
By \ref{ser85} D), $L_{[0]}\subseteq M_{[0]}$ is a free sub-$R_{[0]}$-module, and $M_{[0]}$ has a set of generators of cardinality $n$. Thus, \ref{ser85} D) and \cite[VII.3 Lemme 1]{a} imply that $\rk_R(L)=\rk_{R_{[0]}}(L_{[0]})\leq n$ as claimed.
\end{proof}

\begin{thm}\label{a80}
The following statements are equivalent:
\begin{mylist}{3.9}
\item[(i)] $R$ is principal or $R=0$;
\item[(ii)] Every graded ideal of $R$ is free;
\item[(iii)] Every graded ideal of $R$ is free of rank at most $1$;
\item[(iv)] Graded sub-$R$-modules of free $G$-graded $R$-modules are free;
\item[(v)] Graded sub-$R$-modules of free $G$-graded $R$-modules of finite type are free.
\end{mylist}
\end{thm}

\begin{proof}
If $R=0$, then this is obvious, so we suppose that $R\neq 0$. ``(i)$\Rightarrow$(ii)'': Entirety of $R$ implies that every homogeneous element of $R\setminus 0$ is free, hence principality of $R$ implies that every graded ideal of $R$ is free. ``(ii)$\Rightarrow$(iii)'': As $R$ is monogeneous, this follows from \ref{a60}. ``(iii)$\Rightarrow$(iv)'': By \ref{a50}, a graded sub-$R$-module of a free $G$-graded $R$-module is a direct sum of shifts of graded ideals of $R$, hence a direct sum of free graded $R$-modules and thus itself free. ``(iv)$\Rightarrow$(v)'': Obvious. ``(v)$\Rightarrow$(i)'': As $R$ is free of rank $1$, it follows from \ref{a60} that graded ideals of $R$ are free of rank at most $1$, hence monogeneous. If $x\in R\setminus 0$ is homogeneous, then the graded ideal $\langle x\rangle_R\subseteq R$ is free, hence $x$ is free. Therefore, $R$ is entire and thus principal.
\end{proof}

\begin{cor}\label{a90}
Suppose that $R$ is principal, and let $M$ be a $G$-graded $R$-module.

a) If $M$ is projective, then it is free.

b) $M$ is free if and only if $M\ps$ is so.
\end{cor}

\begin{proof}
a) follows from \ref{a80} since a projective $G$-graded $R$-module is a graded sub-$R$-modules of a free $G$-graded $R$-module. b) follows from a) and \ref{a30}.
\end{proof}


\section{Superfluous monomorphisms}

\begin{no}
A) A monomorphism $u\colon M\rightarrowtail N$ in $\md(R)$ is called \textit{superfluous} if when\-ever $v\colon L\rightarrowtail N$ is a monomorphism in $\md(R)$ with $\im(u)+\im(v)=N$, then $\im(v)=N$.\smallskip

B) Let $u\colon M\rightarrowtail N$ be a monomorphism in $\md(R)$. If $u\ps$ is superfluous, then so is $u$. Indeed, if $v\colon L\rightarrowtail N$ is a monomorphism with $\im(u)+\im(v)=N$, then $\im(u\ps)+\im(v\ps)=(\im(u)+\im(v))\ps=N\ps$. If $u\ps$ is superfluous, then it follows that $\im(v)\ps=\im(v\ps)=N\ps$ and hence $\im(v)=N$. Therefore, $u$ is superfluous.
\end{no}

\begin{prop}\label{super10}
If $\psi$ is not an isomorphism, then there exist a $G$-graded ring $R$ and a superfluous monomorphism $u$ in $\md(R)$ such that $u\ps$ is not superfluous.
\end{prop}

\begin{proof}
If $\psi$ is not an isomorphism, then there exists $g\in\ke(\psi)\setminus 0$. Let $R=K[X]$ be the polynomial algebra over a field $K$ in one variable $X$, furnished with the $G$-graduation with $\deg(X)=g$. We show that the canonical injection $u\colon\langle X\rangle\hookrightarrow R$ has the desired properties. If $\ia\subseteq R$ is a graded ideal with $\langle X\rangle+\ia=R$, then there exist $r\in R$ and $f\in\ia$ with $1=rX+f$. As $(rX)_0=0$ and $\ia$ is graded, taking components of degree $0$ yields $1=f_0\in\ia$, hence $\ia=R$. Therefore, $u$ is superfluous. On the other hand, $X+1\in(R\ps)_0$, hence $\langle X+1\rangle\ps\subsetneqq R\ps$ is a graded ideal. As $\langle X\rangle\ps+\langle X+1\rangle\ps=R\ps$ it follows that $u\ps$ is not superfluous.
\end{proof}


\section{Homological dimensions}

\noindent\textit{Throughout this section, let $\C$ and $\D$ be abelian categories.}

\begin{no}
If $A\in\ob(\C)$, then \[\dimp(A)\dfgl\inf\{n\in\N\mid\exists\text{ projective resolution of }A\text{ of length }n\}\in\N\cup\{\infty\}\] and \[\dimi(A)\dfgl\inf\{n\in\N\mid\exists\text{ injective resolution of }A\text{ of length }n\}\in\N\cup\{\infty\}\] are called \textit{the projective dimension of $A$} and \textit{the injective dimension of $A$,} resp.
\end{no}

\begin{prop}[Schanuel's Lemma]\label{dim10}
Let $A\in\ob(\C)$, let $n\in\N$, and let \[0\longrightarrow K\longrightarrow P_{n-1}\xrightarrow{p_{n-2}}P_{n-2}\longrightarrow\cdots\longrightarrow P_1\overset{p_0}\longrightarrow P_0\overset{\alpha}\longrightarrow A\longrightarrow 0\] and \[0\longrightarrow L\longrightarrow Q_{n-1}\xrightarrow{q_{n-2}}Q_{n-2}\longrightarrow\cdots\longrightarrow Q_1\overset{q_0}\longrightarrow Q_0\overset{\beta}\longrightarrow A\longrightarrow 0\] be exact sequences in $\C$ where $P_i$ and $Q_i$ are projective for every $i\in[0,n-1]$. Then, \[K\oplus Q_{n-1}\oplus P_{n-2}\oplus Q_{n-3}\oplus\cdots\oplus P_0\cong L\oplus P_{n-1}\oplus Q_{n-2}\oplus P_{n-3}\oplus\cdots\oplus Q_0\] if $n$ is even, and \[K\oplus Q_{n-1}\oplus P_{n-2}\oplus Q_{n-3}\oplus\cdots\oplus Q_0\cong L\oplus P_{n-1}\oplus Q_{n-2}\oplus P_{n-3}\oplus\cdots\oplus P_0\] if $n$ is odd.
\end{prop}

\begin{proof}
If $n=0$, then $K\cong A\cong L$ as claimed. Let $n=1$. Denoting by $R$ the fibre product of $\alpha$ and $\beta$, we get a commutative diagram \[\xymatrix{&&0&0&\\0\ar[r]&K\ar[r]&P_0\ar[u]\ar[r]^\alpha&A\ar[u]\ar[r]&0\\0\ar[r]&\ke(\alpha')\ar[u]_{\beta''}\ar[r]&R\ar[u]_{\beta'}\ar[r]^{\alpha'}&Q_0\ar[u]_\beta\ar[r]&0\\&&\ke(\beta')\ar[u]\ar[r]^(.6){\alpha''}&L\ar[u]&\\&&0\ar[u]&0\ar[u]&}\] with exact rows and columns such that $\alpha''$ and $\beta''$ are isomorphisms (\cite[08N3, 08N4]{stacks}). Projectivity of $P_0$ and $Q_0$ implies that $K\oplus Q_0\cong\ke(\alpha')\oplus Q_0\cong R\cong\ke(\beta')\oplus P_0\cong L\oplus P_0$ as desired. Let $n>1$, and suppose the claim to hold for strictly smaller values of $n$. We have exact sequences \[0\longrightarrow\ke(\alpha)\longrightarrow P_0\longrightarrow A\longrightarrow 0\quad\text{and}\quad 0\longrightarrow\ke(\beta)\longrightarrow Q_0\longrightarrow A\longrightarrow 0\] where $P_0$ and $Q_0$ are projective, hence $\ke(\alpha)\oplus Q_0\cong\ke(\beta)\oplus P_0$. Moreover, we have exact sequences \[0\longrightarrow K\longrightarrow P_{n-1}\longrightarrow\cdots\longrightarrow P_2\overset{p_1'}\longrightarrow P_1\oplus Q_0\overset{p_0'}\longrightarrow\ke(\alpha)\oplus Q_0\longrightarrow 0\] and \[0\longrightarrow L\longrightarrow Q_{n-1}\longrightarrow\cdots\longrightarrow Q_2\overset{q_1'}\longrightarrow Q_1\oplus P_0\overset{q_0'}\longrightarrow\ke(\beta)\oplus P_0\longrightarrow 0\] where $p_1'$, $q_1'$, $p_0'$ and $q_0'$ are induced by $(p_1,0)$, $(q_1,0)$, $(p_0,\Id_{Q_0})$ and $(q_0,\Id_{P_0})$, resp. Moreover, $P_i$ and $Q_i$ for $i\in[2,n-1]$ as well as $P_1\oplus Q_0$ and $Q_1\oplus P_0$ are projective. As $\ke(\alpha)\oplus Q_0\cong\ke(\beta)\oplus P_0$ we thus get the desired isomorphisms, and then the claim follows by induction.
\end{proof}

\begin{thm}\label{dim20}
Let $A\in\ob(\C)$, and let $n\in\N$. The following are equivalent:
\begin{aufz}
\item[(i)] $\dimp(A)\leq n$;
\item[(ii)] $A$ has a projective resolution, and if $0\rightarrow K\rightarrow P_{n-1}\rightarrow\cdots\rightarrow P_0\rightarrow A\rightarrow 0$ is an exact sequence in $\C$ where $P_i$ is projective for every $i\in[0,n-1]$, then $K$ is projective.
\end{aufz}
\end{thm}

\begin{proof}
``(i)$\Rightarrow$(ii)'': Finiteness of $\dimp(A)$ implies that $A$ has a projective resolution. Let \[0\longrightarrow K\longrightarrow P_{n-1}\longrightarrow\cdots\longrightarrow P_0\longrightarrow A\longrightarrow 0\] be an exact sequence where $P_i$ is projective for every $i\in[0,n-1]$. By (i) there exists an exact sequence \[0\longrightarrow Q_n\longrightarrow Q_{n-1}\longrightarrow\cdots\longrightarrow Q_0\longrightarrow A\longrightarrow 0\] where $Q_i$ is projective for every $i\in[0,n]$, hence $K$ is a direct summand of a projective object by \ref{dim10} and thus itself projective. ``(ii)$\Rightarrow$(i)'': There exists an exact sequence \[0\longrightarrow K\longrightarrow P_{n-1}\longrightarrow\cdots\longrightarrow P_0\longrightarrow A\longrightarrow 0\] where $P_i$ is projective for every $i\in[0,n-1]$, hence $K$ is projective by (ii), thus $A$ has a projective resolution of length $n$, and therefore $\dimp(A)\leq n$.
\end{proof}

\begin{prop}\label{dim30}
Let $F\colon\C\rightarrow\D$ be an exact functor. 

a) $F$ respects projectivity if and only if $\dimp(A)\geq\dimp(F(A))$ for every $A\in\ob(\C)$.

b) If $A\in\ob(\C)$ has a projective resolution and $F$ respects and reflects projectivity, then $\dimp(A)=\dimp(F(A))$.
\end{prop}

\begin{proof}
a) If $F$ respects projectivity, then $F$ maps a projective resolution of $A\in\ob(\C)$ of length $n$ to a projective resolution of $F(A)$ of length $n$, implying $\dimp(A)\geq\dimp(F(A))$. Conversely, if $A\in\ob(\C)$ is projective while $F(A)$ is not, then $\dimp(A)=0<\dimp(F(A))$.

b) By a) it suffices to show that $n\dfgl\dimp(F(A))\geq\dimp(A)$. If $n=\infty$, then this is clear. If $n<\infty$, then there exists an exact sequence \[P_{n-1}\xrightarrow{p_{n-1}}\cdots\longrightarrow P_0\longrightarrow A\longrightarrow 0\] where $P_i$ is projective for every $i\in[0,n-1]$. Applying $F$ yields an exact sequence \[F(P_{n-1})\xrightarrow{F(p_{n-1})}\cdots\longrightarrow F(P_0)\longrightarrow F(A)\longrightarrow 0\] where $F(P_i)$ is projective for every $i\in[0,n-1]$, thus $F(\ke(p_{n-1}))=\ke(F(p_{n-1}))$ is projective by \ref{dim20}, and therefore $\dimp(A)\leq n$ as desired.
\end{proof}

\begin{cor}\label{dim40}
Suppose that $\C$ has enough projectives, and let $F\colon\C\rightarrow\D$ be an exact functor. Then, $F$ respects and reflects projectivity if and only if $\dimp(A)=\dimp(F(A))$ for every $A\in\ob(\C)$.
\end{cor}

\begin{proof}
If $A\in\ob(\C)$ is not projective but $F(A)$ is so, then $\dimp(A)>0=\dimp(F(A))$. So, the claim follows from \ref{dim30}.
\end{proof}

\begin{no}\label{dim50}
Propositions \ref{dim10}, \ref{dim20}, \ref{dim30} and \ref{dim40} can be dualised to yield corresponding results about injective dimensions.
\end{no}

\begin{no}\label{dim60}
A) A \textit{cogenerator of\/ $\C$} is an $E\in\ob(\C)$ such that the contravariant functor $\hm{\C}{\bullet}{E}\colon\C\rightarrow\ab$ is faithful.\smallskip

B) Let $E$ be an injective cogenerator of $\C$, and let $u\colon A\rightarrow B$ be a morphism in $\C$. Then, $u$ is a monomorphism if and only if $\hm{\C}{u}{E}$ is an epimorphism. Indeed, injectivity of $E$ implies that $\hm{\C}{\bullet}{E}$ turns monomorphisms into epimorphisms. Conversely, suppose that $\hm{\C}{u}{E}$ is an epimorphism, let $C\in\ob(\C)$, and let $v,w\in\hm{\C}{C}{A}$ with $u\circ v=u\circ w$. Then, $\hm{\C}{v}{E}\circ\hm{\C}{u}{E}=\hm{\C}{w}{E}\circ\hm{\C}{u}{E}$, hence $\hm{\C}{v}{E}=\hm{\C}{w}{E}$, and faithfulness of $\hm{\C}{\bullet}{E}$ implies that $v=w$. Thus, $u$ is a monomorphism.
\end{no}

\begin{no}\label{dim70}
A) For a $G$-graded $R$-module $M$ we have a contravariant functor \[\hm{R}{\bullet}{M}\dfgl\bigoplus_{g\in G}\hm{\md(R)}{\bullet}{M(g)}\] that is exact if and only if $M$ is injective (\cite[A.I.2.4]{nvo1}).\smallskip

B) If $N$ is a further $G$-graded $R$-module, then by \cite[2.5 B)]{gcor} there is a canonical isomorphism of contravariant functors \[\hm{R}{\bullet\otimes_RM}{N}\cong\hm{R}{\bullet}{\hm{R}{M}{N}}.\]
\end{no}

\begin{prop}\label{dim80}
Let $E$ be an injective cogenerator of $\md(R)$.

a) The contravariant functor $\hm{R}{\bullet}{E}\colon\md(R)\rightarrow\md(R)$ is faithful and exact.

b) A morphism $u$ of $G$-graded $R$-modules is a monomorphism if and only if $\hm{R}{u}{E}$ is an epimorphism.
\end{prop}

\begin{proof}
a) Exactness holds by \ref{dim70} A). If $g\in G$, then $\bullet(g)$ is an automorphism of $\md(R)$, hence $E(g)$ is an injective cogenerator, and thus $\hm{\md(R)}{\bullet}{E(g)}$ is faithful. As direct sums are faithful, it follows that $\hm{R}{\bullet}{E}$ is faithful.

b) The morphism $\hm{R}{u}{E}=\bigoplus_{g\in G}\hm{\md(R)}{u}{E(g)}$ is an epimorphism if and only if $\hm{\md(R)}{u}{E(g)}$ is so for every $g\in G$, hence the claim follows from \ref{dim60} B).
\end{proof}

\begin{prop}[Lambek's Lemma]\label{dim90}
Let $E$ be an injective cogenerator of $\md(R)$, and let $M$ be a $G$-graded $R$-module. Then, $M$ is flat if and only if $\hm{R}{M}{E}$ is injective.
\end{prop}

\begin{proof}
The $G$-graded $R$-module $M$ is flat if and only if $u\otimes_RM$ is a monomorphism for every monomorphism $u$, hence if and only if $\hm{R}{u\otimes_RM}{E}$ is an epimorphism for every monomorphism $u$ (\ref{dim80} b)), thus if and only if $\hm{R}{u}{\hm{R}{M}{E}}$ is an epimorphism for every monomorphism $u$ (\ref{dim70} B)), and therefore if and only if $\hm{R}{M}{E}$ is injective (\ref{dim70} A)).
\end{proof}

\begin{no}
If $M$ is a $G$-graded $R$-module, then \[\dimpl(M)\dfgl\inf\{n\in\N\mid\exists\text{ flat resolution of }M\text{ of length }n\}\in\N\cup\{\infty\}\] is called \textit{the flat dimension of $A$} (and sometimes also \textit{the weak dimension of $M$}).
\end{no}

\begin{lemma}\label{dim100}
Let $E$ be an injective cogenerator of $\md(R)$, and let $M$ be a $G$-graded $R$-module. Then, $\dimi(\hm{R}{M}{E})\leq\dimpl(M)$.
\end{lemma}

\begin{proof}
If $n\in\N$, then $\hm{R}{\bullet}{E}$ turns a flat resolution of $M$ of length $n$ into an injective resolution of $\hm{R}{M}{E}$ of length $n$ (\ref{dim80}, \ref{dim90}).
\end{proof}

\begin{thm}\label{dim110}
Let $M$ be a $G$-graded $R$-module, and let $n\in\N$. The following are equivalent:
\begin{aufz}
\item[(i)] $\dimpl(M)\leq n$;
\item[(ii)] If $0\rightarrow K\rightarrow P_{n-1}\rightarrow\cdots\rightarrow P_0\rightarrow M\rightarrow 0$ is an exact sequence of $G$-graded $R$-modules where $P_i$ is flat for every $i\in[0,n-1]$, then $K$ is flat.
\end{aufz}
\end{thm}

\begin{proof}
``(i)$\Rightarrow$(ii)'': Let \[0\longrightarrow K\longrightarrow P_{n-1}\longrightarrow\cdots\longrightarrow P_0\longrightarrow M\longrightarrow 0\] be an exact sequence where $P_i$ is flat for every $i\in[0,n-1]$. By \cite[p.\ 4]{nvo1} and \cite[p.\ 135]{tohoku} there exists an injective cogenerator $E$ of $\md(R)$. Applying $\hm{R}{\bullet}{E}$ yields an exact sequence \[0\rightarrow\hm{R}{M}{E}\rightarrow\hm{R}{P_0}{E}\rightarrow\cdots\rightarrow\hm{R}{P_{n-1}}{E}\rightarrow\hm{R}{K}{E}\rightarrow 0\] where $\hm{R}{P_i}{E}$ is injective for every $i\in[0,n-1]$ (\ref{dim80}, \ref{dim90}). Now, we have $\dimi(\hm{R}{M}{E})\leq n$ (\ref{dim100}), hence $\hm{R}{K}{E}$ is injective (\ref{dim20}, \ref{dim50}), and thus $K$ is flat (\ref{dim90}). ``(ii)$\Rightarrow$(i)'': There exists an exact sequence \[0\longrightarrow K\longrightarrow P_{n-1}\longrightarrow\cdots\longrightarrow P_0\longrightarrow M\longrightarrow 0\] where $P_i$ is flat for every $i\in[0,n-1]$. Now, (ii) implies that $K$ is flat, hence $M$ has a flat resolution of length $n$, and thus $\dimpl(M)\leq n$.
\end{proof}

\begin{prop}\label{dim120}
a) If $M$ is a $G$-graded $R$-module, then $\dimp(M)=\dimp(M\ps)$ and $\dimpl(M)=\dimpl(M\ps)$.

b) $\ke(\psi)$ is finite if and only if $\dimi(M)=\dimi(M\ps)$ for every $G$-graded $R$-module $M$.
\end{prop}

\begin{proof}
The category $\md(R)$ is abelian and has enough projectives and injectives (\cite[pp.\ 5f.]{nvo1}). Moreover, $\bullet\ps$ is exact and respects and reflects projectivity and flatness (\cite[A.I.2.2, A.I.2.18]{nvo1}); it respects and reflects injectivity if and only if $\ke(\psi)$ is finite (\cite[2.3]{cihf}). Thus, the claims about $\dimp$ and $\dimi$ follow from \ref{dim40}. Furthermore, $\bullet\ps$ turns a flat resolution of $M$ of length $n$ into a flat resolution of $M\ps$ of length $n$, implying $n\dfgl\dimpl(M\ps)\leq\dimpl(M)$. If $n=\infty$, then the remaining claim is clear. If $n<\infty$, then there exists an exact sequence \[P_{n-1}\xrightarrow{p_{n-1}}\cdots\longrightarrow P_0\longrightarrow M\longrightarrow 0\] in $\md(R)$ where $P_i$ is flat for every $i\in[0,n-1]$. Applying $\bullet\ps$ yields an exact sequence \[(P_{n-1})\ps\xrightarrow{(p_{n-1})\ps}\cdots\longrightarrow (P_0)\ps\longrightarrow M\ps\longrightarrow 0\] in $\md(R\ps)$ where $(P_i)\ps$ is flat for every $i\in[0,n-1]$, hence $\ke(p_{n-1})\ps=\ke((p_{n-1})\ps)$ is flat (\ref{dim110}), thus $\ke(p_{n-1})$ is flat, too, and thus $\dimpl(M)\leq n$.
\end{proof}


\noindent\textbf{Acknowledgement:} I am grateful to Uriya First for showing me how to use Schanuel's Lemma to prove \ref{dim20}.

\end{document}